\documentclass[12 pt, english]{article}
\usepackage[T1]{fontenc}
\usepackage{ulem}
\usepackage[utf8]{luainputenc}
\usepackage{geometry}
\usepackage[pdftex]{graphicx}
\usepackage{amstext}
\usepackage{amsthm}
\usepackage{amssymb}
\usepackage{amsmath}
\usepackage[T1,T2A]{fontenc}
\usepackage[utf8]{inputenc}
\usepackage[main=english,russian]{babel}
\usepackage{setspace}
\usepackage{esint}
\usepackage{comment}
\usepackage{babel}
\usepackage{float}
\usepackage{amsfonts}
\usepackage{fullpage} 
\usepackage{parskip} 
\usepackage{tikz} 
\usepackage{indentfirst}
\renewcommand{\:}{\colon}

\def\dis{\operatorname{dis}}
\def\diam{\operatorname{diam}}
\title{Gromov–-Hausdorff distance between vertex sets of regular polygons inscribed in a given circle}
\author{Talant Talipov}
\date{}
\begin{document}
\thispagestyle{empty}
\newtheorem{Th}{Theorem}[section]
\newtheorem{Lem}{Lemma}[section]
\newtheorem{Rem}{Remark}[section]
\newtheorem{Cor}{Corollary}[section]
\theoremstyle{definition}
\newtheorem{Ex}{Example}[section]
\newtheorem{Def}{Definition}[section]
\newtheorem{Prep}{Proposition}[section]
\sloppy

\maketitle

\begin{abstract}
We calculate the Gromov--Hausdorff distance between vertex sets of regular polygons endowed with the round metric. We give a full answer for the case of $n$- and $m$-gons with $m$ divisible by $n$. Also, we calculate all distances to $2$-gons and $3$-gons
\end{abstract}

\section{Introduction}
In the present paper we study the class of all metric spaces considered up to isometry, endowed with the Gromov--Hausdorff distance. Note that the exact values of the Gromov--Hausdorff distances between specific metric spaces are known only for a small number of cases. For example, in~\cite{IvaTuzSimpDist}, the Gromov–-Hausdorff distances from a broad class of metric spaces to 1-spaces, i.e. metric spaces with one non-zero distance, were calculated. In ~\cite{YiboTuz}, the Gromov–-Hausdorff distance between a segment and a circle with intrinsic metric was obtained. In ~\cite{Memoli}, the Gromov–-Hausdorff distances between spheres of different dimensions were calculated in some cases and estimated in other ones. Also, the Gromov--Hausdorff distance between the vertex set of regular polygons and circle were calculated, as well as between different regular  $m$- and $(m + 1)$-gons inscribed in the same circle. In the present paper we extend the results of~\cite{Memoli} for the case of $m$- and $n$-gons, provided that $m$ is divisible by $n$. Also we have calculated all distances to $2$-gons and $3$-gons. The author expresses his gratitude to his supervisor, Prof. Alexey A. Tuzhilin, and to Prof. Alexander O. Ivanov for posing the problem and help in the work.

\section{Preliminaries}
\label{Before}
Let $X$ be an arbitrary metric space. The distance between points $x, y$ we denote by $d(x,y)$ or $|xy|$. For any non-empty $A$, $B$ the \textit{Hausdorff distance\/} between $A$ and $B$ is defined as follows$\/\colon$
\begin{equation*}
    d_{H}(A,B) = \max\bigl\{\sup\limits_{a \in A}\inf\limits_{b \in B}d(a,b), \sup\limits_{b \in B}\inf\limits_{a \in A}d(a,b)\bigr\}.
\end{equation*}
Let $X$ and $Y$ be metric spaces. If $X'$ and $Y'$ are subsets of a metric space $Z'$, provided $X'$ is isometric to $X$, and $Y'$ is isometric to $Y$, then $(X', Y', Z')$ is called \textit{realization} of the pair $(X,Y)$. \textit{The Gromov--Hausdorff distance\/} between $X$ and $Y$ is the value
\begin{equation*}
    d_{GH}(X, Y) = \inf\bigl\{r: \exists (X',Y',Z'), d_{H}(X',Y') \leqslant r\bigr\}.
\end{equation*}
\begin{Def}
Given two sets $X$ and $Y$, a \textit{correspondence\/} between $X$ and $Y$ is a subset $R \subset X \times Y$ such that for any $x \in X$ there exists $y \in Y$ with $(x, y) \in R$ and, vise versa, for any $y \in Y$ there exists $x \in X$ with $(x, y) \in R$. If $X$, $Y$ are metric spaces, then we define \textit{the correspondence distortion\/} $R$ as follows$\/\colon$ $\dis R = \sup\Bigl\{\bigl||x_{1}x_ {2}| - |y_{1}y_{2}|\bigr|: (x_{1},y_{1}), (x_{2},y_{2}) \in R\Bigr\}$.
\end{Def}
We denote by $\mathcal{R}(X, Y)$ the set of all correspondences between $X$ and $Y$.
\begin{Th}[\cite{BurBurIva}]
Let $X$ and $Y$ be metric spaces. Then
\begin{equation*}
    d_{GH}(X,Y) = \frac{1}{2}\inf\bigl\{\dis R: R \in \mathcal{R}(X,Y)\bigr\}.
\end{equation*}
\end{Th}
In~\cite{Memoli} the following construction was considered. For a metric space $(X,d_{X})$, we define the pseudo-ultrametric space $(X,u_{X})$ where $u_{X}\colon X\times X\to \mathbb{R}_{+}$ is defined by
\begin{equation*}
    u_{X}\: (x,y)\mapsto \inf\bigl\{\max\limits_{0\leqslant i \leqslant n-1} d_{X}(x_{i},x_{i+1}):x_{0} = x,...,x_{n} = y\bigr\}.
\end{equation*}
Now, define $\mathbf{U}(X)$ to be the quotient metric space of $(X,u_{X})$ over the equivalence $x \sim y$ if and only if $u_{X}(x,y) = 0$.
\begin{Th}[\cite{Memoli}]
\label{MemoliTh}
For all bounded metric spaces $X$ and $Y$, it holds
\begin{equation*}
    d_{GH}(X,Y) \geqslant d_{GH}\bigl(\mathbf{U}(X), \mathbf{U}(Y)\bigr).
\end{equation*}
\end{Th}
\begin{Def}
By \textit{simplex\/} we call a metric space, all whose non-zero distances equal to each other. If $m$ is an arbitrary cardinal number, then by $\lambda \Delta_{m}$ we denote a simplex containing $m$ points and such that all its non-zero distances equal to $\lambda$.
\end{Def}
Let $X$ be an arbitrary set consisting of more than one point, $2 \leqslant m \leqslant \#X$ a cardinal number,
and $\lambda > 0$. By $\mathcal{D}_{m}(X)$ we denote the family of all possible partitions of the set X into m
non-empty subsets.  For any non-empty $A, B \subset X$, we put $|AB| = \inf\bigl\{|ab|: a \in A, b \in B\bigr\}$. Now let $X$ be a metric space. Then for each $D = \{X_{i}\}_{i \in I} \in \mathcal{D}_{m}(X)$ we put
\begin{equation*}
    \diam D = \sup\limits_{i \in I} \diam X_{i}\/;
\end{equation*}
\begin{equation*}
    \alpha(D) = \inf\bigl\{|X_{i}X_{j}|: i \neq j\bigr\}.
\end{equation*}
\begin{Th}[\cite{IvaTuzSimpDist}]
\label{SimplexTh}
Let $X \neq \Delta_{1}$ be an arbitrary metric space, $2 \leqslant m \leqslant \#X$ a cardinal
number, and $\lambda > 0$. Then
\begin{equation*}
    2d_{GH}(\lambda\Delta, X) = \inf\limits_{D \in \mathcal{D}_{m}(X)} \max\bigl\{\diam D, \lambda - \alpha(D), \diam X - \lambda\bigr\}.
\end{equation*}
\end{Th}

\section{Gromov--Hausdorff distance between vertex sets of regular polygons inscribed in a given circle}
For each integer $n \geqslant 2$, let $P_{n}$ be the set of vertices of a regular $n$-gon inscribed in the unit circle $S^1$. Notice that $P_{2}$ is the pair of diametrically opposed points. Furthermore, we endow $P_{n}$ with the restriction of the geodesic distance on $S^1$. For $m,n \geqslant 2$, we define the values $p_{m,n} = d_{GH}(P_{m}, P_{n})$.
\begin{Prep}[\cite{Memoli}]
\label{MemoliPrep}
For all $m \geqslant 2$, we have
\begin{equation*}
    d_{GH}(S^1, P_{m}) = \frac{\pi}{m};
\end{equation*}
\begin{equation*}
    d_{GH}(P_{m}, P_{m+1}) = \frac{\pi}{m + 1}.
\end{equation*}
\end{Prep}
The next technical result will be used in what follows.
\begin{Lem}
\label{MainLemma}
Let $n \geqslant 2$, $p \in \mathbb{N}$. Then for any $i, j = 1, 2, ..., n$ and $k,l = 1, 2, ..., p - 1$,
\begin{equation*}
	\Bigl|\min\bigl(|i - j|, n - |i - j|\bigr) - \min\bigl(\bigl|i - j + \frac{k - l}{p}\bigr|, n - \bigl|i - j + \frac{k - l}{p}\bigr|\bigr)\Bigr| \leqslant \frac{|k - l|}{p}.
\end{equation*}
\begin{proof}
Put $S = \Bigl|\min\bigl(|i - j|, n - |i - j|\bigr) - \min\bigl(\bigl|i - j + \frac{k - l}{p}\bigr|, n - \bigl|i - j + \frac{k - l}{p}\bigr|\bigr)\Bigr|$. Since $\frac{|k - l|}{p} < 1$, we have
\begin{equation*}
    |i - j + \frac{k - l}{p}| = 
	\begin{cases}
    i - j + \frac{k - l}{p}, & \text{if}\ i - j > 0;\\
    \frac{|k - l|}{p}, & \text{if}\ i - j = 0;\\
    j - i + \frac{l - k}{p}, & \text{if}\ i - j < 0.
    \end{cases}
\end{equation*}
If $i - j = 0$ then
\begin{equation*}
    S = \Bigl|\min\bigl(0, n\bigr) - \min\bigl(\frac{|k - l|}{p}, n - \frac{|k - l|}{p}\bigr)\Bigr| = \frac{|k - l|}{p}.
\end{equation*}
Without loss of generality, we will assume that $i - j > 0$. Let us consider a few cases. \par
1) Assume that $i - j < \frac{n}{2}$. If $i - j + \frac{k - l}{p} > \frac{n}{2}$ then $i - j = \frac{n - 1}{2}$ and $\frac{k - l}{p} > \frac{1}{2}$. In this case,
\begin{equation*}
    S = \Bigl|\frac{n - 1}{2} - \bigl(n - (\frac{n - 1}{2} + \frac{k - l}{p})\bigr)\Bigr| = \Bigl|\frac{k - l}{p} - 1\Bigr| < \frac{|k - l|}{p},
\end{equation*}
where the last inequality holds because $\frac{k - l}{p} > \frac{1}{2}$. If $i - j + \frac{k - l}{p} \leqslant \frac{n}{2}$ then
\begin{equation*}
    S = \Bigl|i - j - \bigr(i - j + \frac{k - l}{p}\bigr)\Bigr| = \frac{|k - l|}{p}.
\end{equation*} \par
2) Assume that $i - j = \frac{n}{2}$. Then
\begin{equation*}
    S = \Bigl|\frac{n}{2} - \min\bigl(\frac{n}{2} + \frac{k - l}{p}, \frac{n}{2} + \frac{l - k}{p}\bigr)\Bigr| = \frac{|k - l|}{p}.
\end{equation*} \par
3) Assume that $i - j > \frac{n}{2}$. If $i - j + \frac{k - l}{p} < \frac{n}{2}$ then $i - j = \frac{n + 1}{2}$ and $\frac{l - k}{p} > \frac{1}{2}$. In this case,
\begin{equation*}
    S = \Bigl|\frac{n - 1}{2} - \bigl(\frac{n + 1}{2} + \frac{k - l}{p}\bigr)\Bigr| = \Bigl|\frac{l - k}{p} - 1\Bigr| < \frac{|k - l|}{p},
\end{equation*}
where the last inequality holds because $\frac{l - k}{p} > \frac{1}{2}$. If $i - j + \frac{k - l}{p} \geqslant \frac{n}{2}$ then
\begin{equation*}
    S = \Bigl|n - i + j - \bigr(n - i + j - \frac{k - l}{p}\bigr)\Bigr| = \frac{|k - l|}{p}.
\end{equation*}
The proof is completed.
\end{proof}
\end{Lem}
Now let us formulate the main results of this paper.
\begin{Th}
\label{MainTh}
Let $2 \leqslant n \leqslant m$ and $m$ is divisible by $n$. Then
\begin{equation*}
    p_{n,m} = \frac{\pi}{n} - \frac{\pi}{m}.
\end{equation*}
\end{Th}
\begin{proof}
Let $u_{1},...,u_{n}$ be the vertices of $P_{n}$ and $v_{1}, ..., v_{m}$ be the vertices of $P_{m}$. Let us prove that $p_{n,m} \geqslant \frac{\pi}{n} - \frac{\pi}{m}$. By Theorem~\ref{MemoliTh},
\begin{equation*}
    p_{n,m} \geqslant d_{GH}\bigl(\mathbf{U}(P_{m}), \mathbf{U}(P_{n})\bigr).
\end{equation*}
Notice that $\mathbf{U}(P_{m}), \mathbf{U}(P_{n})$ are simpexes with distances $\frac{2\pi}{m}$ и $\frac{2\pi}{n}$ containing $m$ and $n$ points respectively. By Theorem~\ref{SimplexTh},
\begin{equation*}
    p_{n,m} \geqslant d_{GH}\bigl(\mathbf{U}(P_{m}), \mathbf{U}(P_{n})\bigr) = \frac{1}{2} \max\Bigl\{\frac{2\pi}{m}, \frac{2\pi}{n} - \frac{2\pi}{m}\Bigr\} \geqslant \frac{\pi}{n} - \frac{\pi}{m}.
\end{equation*}
Let us prove the upper bound. Let $m = pn$, where $p \in \mathbb{N}_{+}$. Then $\frac{\pi}{n} - \frac{\pi}{m} = \frac{(p - 1)\pi}{pn}$. Now, we construct the following correspondence $R \in \mathcal{R}(P_{n}, P_{m})$ with $\dis R \leqslant \frac{2(p - 1)\pi}{pn}$, what completes the proof. Namely, let us put
\begin{equation*}
    R = \bigcup\limits_{i=1}^{n} \bigl\{(u_{i},v_{pi - k}): k = 0, 1,..., p - 1)\bigr\}.
\end{equation*}
Then, in accordance with Lemma~\ref{MainLemma}, for any $i, j = 1, 2, ..., n$ and $k, l = 0, 1, ..., p - 1$,
\begin{equation*}
    \begin{aligned}
    \bigl|d(u_{i}, u_{j}) - d(v_{pi - l}, v_{pj - k})\bigr| = \Bigl|\frac{2\pi}{n}\min\bigl(|i - j|, n - |i - j|\bigr) - \frac{2\pi}{pn}\min\bigl(|pi - pj + k - l|, pn - \\
    |pi - pj + k - l|\bigr)\Bigr| \leqslant \frac{2\pi |k - l|}{pn} \leqslant \frac{2(p - 1)\pi}{pn}.
    \end{aligned}
\end{equation*}
Hence, $\dis R \leqslant \frac{2(p - 1)\pi}{pn}$, which is what was required.
\end{proof}
\begin{Th}
Let $m \geqslant 2$. Then
\begin{equation*}
    p_{2,m} = 
    \begin{cases}
    \frac{\pi}{2} - \frac{\pi}{2m}, & \text{if}\ m \text{ is odd;} \\
    \frac{\pi}{2} - \frac{\pi}{m}, & \text{if}\ m \text{ is even.}
    \end{cases}
\end{equation*}
\end{Th}
\begin{proof}
The case of even $m$ immediately follows from Theorem~\ref{MainTh}. Let $m$ be an odd number. Let $u_{1}, u_{2}$ be the vertices of $P_{2}$ and $v_{1},...,v_{m}$ be  the vertices of $P_{m}$. Notice that $P_2 = \pi\Delta_{2}$. By Theorem~\ref{SimplexTh}:
\begin{equation*}
    d_{GH}(P_{2}, P_{m}) = \frac{1}{2}\inf_{D \in \mathcal{D}_{2}}\max\bigl\{\diam D,\/\pi - \alpha(D),\/\diam P_{m} - \pi\bigr\},
\end{equation*}
where $\mathcal{D}_{2}$ is the family of all possible partitions of $P_{m}$ into $2$ non-empty subsets. In this case we have $\alpha(D) = \frac{2\pi}{m}$ and $\diam P_{m} = \pi - \frac{\pi}{2m}$. Then
\begin{equation*}
d_{GH}(P_{2}, P_{m}) = \frac{1}{2}\inf_{D \in \mathcal{D}_{2}}\max\{\diam D,\/\pi - \frac{\pi}{2m}\}.
\end{equation*}
Let's show that for any $D \in \mathcal{D}_{2}$ it holds
\begin{equation*}
    \diam D \geqslant \pi - \frac{\pi}{2m}.
\end{equation*}
Assume that the reverse is true for the partition $D = \{X_{1}, X_{2}\}$, i.e. $d = \diam D < \pi - \frac{\pi}{2m}$. Without loss of generality, we may assume that $\diam X_{1} = d$. Then there are vertices $v_{i}, v_{j} \in P_{m}$ such that $v_{i}, v_{j} \in X_{1}$ and $d(v_{i}, v_ {j}) = d$. Consider the vertices $v_{k}, v_{l} \in P_{m}$ adjacent to $v_{i}$ and $v_{j}$, respectively, and not lying inside the smaller circular arc of $S^1$ connecting $v_{i}$ with $v_{j}$. The vertex $v_{k}$ cannot belong to $X_{1}$, because otherwise
\begin{equation*}
    d = \diam X_{1} \geqslant d(v_{k}, v_{j}) = d(v_{i}, v_{j}) + \frac{2\pi}{m} = d + \frac{2\pi}{m} > d.
\end{equation*}
Similarly, $v_{l} \in X_{2}$. Then
$$
d = \diam D \geqslant \diam X_{2} \geqslant d(v_{k}, v_{l}) \geqslant d(v_{i}, v_{j}) + \frac{2\pi}{m} = d + \frac{2\pi}{m} > d.
$$
Thus, the theorem is proved.
\end{proof}
\begin{Th}
Let $m \geqslant 3$ and $r$ is the remainder of dividing $m$ by $3$. Then
\begin{equation*}
    p_{3,m} = 
    \begin{cases}
    \frac{\pi}{3} - \frac{\pi}{m}, & \text{if}\ r = 0; \\
    \frac{\pi}{3} - \frac{r\pi}{3m}, & \text{if}\ r \neq 0.
    \end{cases}
\end{equation*}
\end{Th}
\begin{proof}
The case of $r = 0$ immediately follows from Theorem~\ref{MainTh}. Let $r > 0$. Notice that $P_3 = \frac{2\pi}{3}\Delta_{3}$. By Theorem~\ref{SimplexTh},
\begin{equation*}
    d_{GH}(P_{3}, P_{m}) = \frac{1}{2}\inf_{D \in \mathcal{D}_{3}}\max\bigl\{\diam D,\/\frac{2\pi}{3} -\alpha(D),\/\diam P_{m} - \frac{2\pi}{3}\bigr\},
\end{equation*}
where $\mathcal{D}_{3}$ --- the family of all possible partitions of $P_{m}$ into $3$ non-empty subsets.
In this case we have $\alpha(D) = \frac{2\pi}{m}$ and $\diam P_{m} \leqslant \pi$. According to Proposition~\ref{MemoliPrep}, we have $p_{3,4} = \frac{\pi}{4}$. Now let $m \geqslant 5$. In this case we have
\begin{equation*}
    d_{GH}(P_{3}, P_{m}) = \frac{1}{2}\inf_{D \in \mathcal{D}_{3}}\max\{\diam D,\/\frac{2\pi}{3} - \frac{2\pi}{m}\}.
\end{equation*}
Put $q = [\frac{m}{3}]$. Assume that $r = 1$. Consider the following partition $D = \{X_{1}, X_{2}, X_{3}\}\/\colon$
\begin{equation*}
    X_{1} = \{v_{1}, v_{2}, ..., v_{q}\};
\end{equation*}
\begin{equation*}
    X_{2} = \{v_{q + 1}, v_{q + 2}, ..., v_{2q}\};
\end{equation*}
\begin{equation*}
    X_{3} = \{v_{2q + 1}, v_{2q + 2}, ..., v_{m}\}.
\end{equation*}
Then, $\diam D = (\frac{m - 1}{3})\frac{2\pi}{m} = \frac{2\pi}{3} - \frac{2\pi}{3m}$. Let us show that for any $D \in \mathcal{D}_{2}$ we have
\begin{equation*}
    \diam D \geqslant \frac{2\pi}{3} - \frac{2\pi}{3m}.
\end{equation*} \\
Assume that the reverse is true for the partition $D = \{X_{1}, X_{2}, X_{3}\}$, i.e. $d = \diam D < \frac{2\pi}{3} - \frac{2\pi}{3m}$. Without loss of generality, we assume that the set $X_{1}$ contains more than one point. Then there are vertices $v_{i}, v_{j} \in P_{m}$ such that $v_{i}, v_{j} \in X_{1}$ and $d(v_{i}, v_ {j}) = \diam X_{1}$. Let us show that every vertex $v_{k} \in X_{1}$ must lie inside the smaller arc of the circle $S^1$ connecting $v_{i}$ and $v_{j}$. Suppose that $v_{k} \in X_{1}$ lies outside this arc. Then the circle is divided into 3 arcs $v_{i}v_{j}, v_{j}v_{k}, v_{k}v_{i}$. The length of each of them must not be greater than $\diam X_{1} \leqslant d$. Then
\begin{equation*}
    2\pi = d(v_{i}, v_{j}) + d(v_{j}, v_{k}) + d(v_{k}, v_{i}) \leqslant 3d < 2\pi - \frac{2\pi}{m} < 2\pi.
\end{equation*}
Thus, each of the sets $X_{1}$, $X_{2}$ and $X_{3}$ is a set of consecutive vertices of $P_{m}$, and diameter of each $X_{i}$ is at most $\frac{2\pi}{ 3} - \frac{2\pi}{3m} - \frac{2\pi}{m}$. Then
\begin{equation*}
    2\pi - \frac{6\pi}{m} = \diam X_{1} + \diam X_{2} + \diam X_{3} \leqslant 2\pi - \frac{6\pi}{m} - \frac{2\pi}{m} < 2\pi - \frac{6\pi}{m}.
\end{equation*}
Thus, $d_{GH}(P_{3}, P_{m}) = \frac{\pi}{3} - \frac{\pi}{3m}$. Assume that $r = 2$. Consider the following partition $D = \{X_{1}, X_{2}, X_{3}\}\/\colon$
\begin{equation*}
    X_{1} = \{v_{1}, v_{2}, ..., v_{q}\};
\end{equation*}
\begin{equation*}
    X_{2} = \{v_{q + 1}, v_{q + 2}, ..., v_{2q + 1}\};
\end{equation*}
\begin{equation*}
    X_{3} = \{v_{2q + 2}, v_{2[\frac{m}{3}] + 2}, ..., v_{m}\}.
\end{equation*}
Then $\diam D = (\frac{m - 2}{3})\frac{2\pi}{m} = \frac{2\pi}{3} - \frac{4\pi}{3m} $. Repeating the arguments for the case $r = 1$, we get that for any $D \in \mathcal{D}_{3}$ we have
\begin{equation*}
    \diam D \geqslant \frac{2\pi}{3} - \frac{4\pi}{3m}.
\end{equation*}
Thus, $d_{GH}(P_{3}, P_{m}) = \frac{\pi}{3} - \frac{2\pi}{3m}$. It completes the proof.
\end{proof}

\end{document}